\documentclass[12pt]{amsart}

\usepackage{amsfonts,amsmath,amssymb,amsthm}
\usepackage{graphicx} 
\usepackage{xcolor}
\usepackage{mathrsfs}
\usepackage{tikz-cd}
\usepackage{pdflscape} 
\usepackage{hyperref}
\usepackage{makecell} 
\usepackage[all]{xy}

\newcommand{\cT}{\mathcal{T}}

\newtheorem{thm}{Theorem}
\newtheorem{dfn}{Definition}

\newtheorem{cor}{Corollary}
\newtheorem*{non-thm}{Theorem}
\newtheorem{prop}{Proposition}
\newtheorem{remark}{Remark}
\newtheorem{d-prop}{Definition-Proposition}

\usepackage{enumerate}

  {\begin{list}{}{\setlength{\leftmargin}{#1}\setlength{\rightmargin}{0cm}}%
   \item[]}{\end{list}}
   
   \usepackage{mdframed}
\newmdenv[
  linewidth=2pt,          
  linecolor=black,        
  topline=false,          
  bottomline=false,       
  rightline=false,        
  leftline=true,          
  innerleftmargin=10pt,   
  leftmargin=15pt,        
  rightmargin=0pt,        
  skipabove=\topsep,      
  skipbelow=\topsep       
]{summary}

\title{On Duality, Legendre Bundles and Deformations}
\author{N.C. Combe, P.G. Combe, H.K. Nencka }

\newcommand{\Addresses}{{
  \bigskip
  \footnotesize

 N.C. Combe, \textsc{De Vinci Research Center, De Vinci Higher Education, Paris, France.}

  \medskip

P.G. Combe, \textsc{Baltic Institute of Mathematics.}
 
  \medskip

H.K. Nencka, \textsc{Baltic Institute of Mathematics.}

}}

\begin{document}
\maketitle
\begin{abstract}
    We introduce the \textit{Legendre bundle}, a geometric structure 
encoding the essential duality of dually flat (Hessian) manifolds, 
and demonstrate that both exponential families in information geometry 
and a natural class of quantum field theories --- which we term 
\textit{Hessian QFTs} --- arise as distinct realisations of this 
single framework. The Legendre bundle is shown to carry a canonical 
para-K\"ahler structure. 
\end{abstract}
\section{Introduction}
\subsection{Motivation and object of study.}
The formalism of {\it Legendre duality} appears in settings of independent importance: the dually flat exponential families of information geometry (see for instance \cite{CM,CCN,Ch}), and the coupling space–one-point functions correspondence of quantum field theory \cite{CV}. No geometric framework currently exists that treats these as instances of a single underlying structure. Our goal is to fill this gap by introducing the {\it Legendre bundle} formalism, a geometric object of which both theories are distinct and natural realisations. Both frameworks rest on the same core data: a convex potential $\Phi$ on a parameter space, its Legendre transform yielding dual coordinates, and an exponential pairing defining a canonical family. The Legendre bundle packages these into a single geometric object.

\subsection{Main result.}
We show that dually flat statistical manifolds and coupling spaces of certain quantum field theories are both natural incarnations of a Legendre bundle. The Legendre bundle is shown to be a para‑Kähler vector bundle.

\subsection{Structure of the paper.}
Section~\ref{S:LegendreBundle} defines the Legendre bundle and establishes its equivalence with dually flat manifolds. 

Section~\ref{S:F-Legendrebundle} introduces families of Legendre bundles over the formal disk $\operatorname{Spec} k[[u]]$, where $k$ is a field of characteristic 0 ($\mathbb{R}$ or $\mathbb{C}$).

Section~\ref{S:Main} states and proves the main theorem (Theorem~\ref{T:1}). In the propositions Prop. \ref{P:ParaKahler}, Prop.~\ref{P:FamilyParaKahler} and the Corollary~\ref{C:QFTParaKahler}, we highlight the natural existence of para-K\"ahler structures on the Legendre bundles.

\section{The Legendre bundle}\label{S:LegendreBundle}
We introduce the notion of a (classical) Legendre bundle show that it captures dually flat (Hessian) manifolds, instances of this include exponential families in statistics.
\subsection{Legendre Duality}

We organise the relevant data into a category in order to make the notion of morphism between Legendre structures precise. The category $\mathbf{CPS}$ of convex potential spaces has as objects pairs $(\mathcal{U}, \Psi)$, where $\mathcal{U} \subseteq E$ is an open convex subset of a finite-dimensional vector space $E$ and $\Psi \in C^{\infty}$, where $\Psi : \mathcal{U} \to \mathbb{R} $
is smooth and strictly convex, and morphisms are affine maps preserving the potential up to an additive affine term.

\smallskip 

For $(\mathcal{U}, \Psi) \in \rm Ob(\mathbf{CPS})$, the Legendre conjugate is defined by
\[
\Psi^*(\eta) = \sup_{v \in \mathcal{U}} \bigl( \langle v, \eta \rangle - \Psi(v) \bigr), 
\qquad \eta \in E^*,
\]
where $E^*$ is the dual of $E$, and 
\[
\langle \cdot, \cdot \rangle : E^* \times E \to \mathbb{R}
\]
is the natural pairing. The image
\[
\mathcal{U}^* = d\Psi(\mathcal{U}) \subseteq E^*
\]
is an open convex set. The Legendre map
\[
v \mapsto \eta = d\Psi(v)
\]
is a diffeomorphism $\mathcal{U} \to \mathcal{U}^*$ satisfying the Fenchel--Young identity
\[
\Psi(v) + \Psi^*(\eta) = \langle v, \eta \rangle, 
\qquad \text{whenever } \eta = d\Psi(v).
\]

\begin{remark}
The category $\mathbf{CPS}$ is a full subcategory of the Hessian manifolds category. Indeed, every $(\mathcal{U}, \Psi)$ carries a Hessian metric given by
\[
g_{ij} = \partial_i \partial_j \Psi,
\] where we use the notation $\partial_i=\frac{\partial}{\partial u_i}$, for some chose coordinate system $(u_1,\cdots,u_n)\in \mathcal{U}$ turning thus \,  $\mathcal{U}$ into a Hessian domain (or manifold). Conversely, any Hessian manifold that is affinely diffeomorphic to an open convex set and admits a global potential belongs to $\mathbf{CPS}$.
\end{remark}

\subsection{Dually Flat Manifolds}

Let $(\Theta,\Psi) \in {\rm Ob}(\mathbf{CPS})$, where we assume that $\Theta\subseteq E=\mathbb{R}^n$. Since $E=\mathbb{R}^n$ we may define a coordinate system $\theta=(\theta_1,\cdots,\theta_n)\in \Theta$ where coordinates $\theta_i$ are the standard linear coordinates on $E$, which are global affine coordinates on $\Theta$. 

\begin{remark}
    Note however that, for an arbitrary dually flat manifold which is not a subset of $\mathbb{R}^n$ the existence of such coordinates is not automatic.
\end{remark}

The affine structure on $\Theta$ inherited from $E=\mathbb{R}^n$ defines a flat, torsion-free connection $\nabla^+$ on $T\Theta$: the one for which the coordinate vector fields $\partial_i$ are parallel.

 The Hessian of $\Psi$ defines a Riemannian metric on $\Theta$
\[
g_{ij}(\theta) = \partial_i \partial_j \Psi(\theta).
\]
The Legendre map $\eta = \nabla \Psi(\theta)$ is a diffeomorphism from $\Theta$ onto its image $\Theta^*$ (which is an open convex subset of the dual vector space $E^*$). The coordinates $\eta_i$ on $\Theta^*$ (the images of the Legendre map) are the linear coordinates inherited from the dual vector space $E^*$: \[
\eta_i = \partial_i \Psi(\theta), \qquad \eta \in \Theta^* \subseteq E^*.
\] 
 The dual coordinates induce an affine structure on $\Theta^*$, yielding a dual connection $\nabla^-$. The pair $(\nabla^+, \nabla^-)$ satisfies the duality condition
\[
Z \, g(X, Y) 
= g(\nabla^+_Z X, Y) + g(X, \nabla^-_Z Y)
\]
for all vector fields $X, Y, Z$ on $\Theta$.

We now consider $\Theta$ as a smooth manifold.
\begin{dfn}
A \emph{dually flat manifold} is a Riemannian manifold $(\Theta, g)$ equipped with a pair of torsion-free flat connections $(\nabla^+, \nabla^-)$ satisfying the duality condition above, and admitting local potential functions $\Psi$ and $\Psi^*$ related by the Fenchel--Young identity.
\end{dfn}

\begin{remark}
The connections $\nabla^+$ and $\nabla^-$ are not independent: each determines the other via the metric $g$. The potential $\Psi$ acts as a generating function for the full geometric structure, encoding both the metric and the dual affine coordinates.
\end{remark}

\subsection{(Classical) Legendre bundle}
\begin{dfn}[Legendre bundle]\label{D:Legendre_Bundle}
Let $B$ be a smooth manifold. A \emph{Legendre bundle} over $B$ is a quintuple
\[
(H,\, \langle\!\langle \cdot, \cdot \rangle\!\rangle,\, \nabla^+,\, \nabla^-,\, \Psi)
\]
where:
\begin{enumerate}
    \item $H = H^+ \oplus H^-$ is a vector bundle over $B$ with isomorphisms 
    \[
    H^+ \cong TB, \qquad H^- \cong T^*B;
    \]
    
    \item $\langle\!\langle \cdot, \cdot \rangle\!\rangle$ is a non-degenerate symmetric bilinear form on $H$ such that \begin{itemize}
    \item $
    \langle\!\langle H^+, H^+ \rangle\!\rangle = 0, \qquad
    \langle\!\langle H^-, H^- \rangle\!\rangle = 0$
    \item  and the induced pairing $H^+ \otimes H^- \to \mathcal{O}_B,$ (where $\mathcal{O}_B$ is the structure sheaf of smooth functions on $B$) coincides with the natural evaluation pairing of a covector on a vector (i.e., for $X \in H^+$, $\alpha \in H^-$, $\langle\!\langle X, \alpha \rangle\!\rangle = \alpha(X)$).
    \end{itemize}
    \item $\nabla^+$ and $\nabla^-$ are flat connections on $H^+$ and $H^-$ respectively, satisfying the duality condition
    \[
    Z\langle\!\langle X, \alpha \rangle\!\rangle 
    =
    \langle\!\langle \nabla^+_Z X, \alpha \rangle\!\rangle 
    +
    \langle\!\langle X, \nabla^-_Z \alpha \rangle\!\rangle
    \]
    for all sections $X \in \Gamma(H^+)$, $\alpha \in \Gamma(H^-)$, and vector fields $Z$ on $B$;

    \item Let $\Psi \in C^\infty(B)$ be a strictly convex potential with respect to the flat structure of $\nabla^+$. The associated Legendre morphism
\[
\mathscr{L}_{\Psi} : H^+ \to H^-
\]
is defined in $\nabla^+$-flat coordinates $(\theta^i)$ by
\[
\mathscr{L}_{\Psi}(\partial_j) = (\partial_i \partial_j \Psi)\, d\theta^i,
\]
and extended linearly. In invariant terms, $\mathscr{L}_{\Psi}$ is the unique bundle morphism such that
\[
\langle\!\langle X,\mathscr{L}_{\Psi}(Y) \rangle\!\rangle
=
\langle\!\langle Y,\mathscr{L}_{\Psi}(X) \rangle\!\rangle, \quad \text{and} \quad \mathscr{L}_{\Psi}(\partial_i)=\sum_j(\partial_i\partial_j \Psi)d\theta^j.\] \end{enumerate}

In local $\nabla^+$-flat coordinates $\theta^i$, the dual coordinates are
\[
\eta_i = \partial_i \Psi,
\]
the pairing satisfies
\[
\langle\!\langle \partial_i, d\theta^j \rangle\!\rangle = \delta_i^j,
\]
and the induced metric is the Hessian metric $g = D^2 \Psi.$
\end{dfn}

\subsection{Equivalence with Dually Flat Manifolds}
We show that the Legendre bundle is precisely a dually flat manifold expressed in bundle form. 
\begin{prop}[Equivalence of structures]\label{P:2.4}
Let $B$ be a smooth manifold. The following two structures are equivalent:

\begin{enumerate}
\item A \emph{dually flat (Hessian) manifold} $(B, g, \nabla^+, \nabla^-)$ together with a strictly convex potential $\Psi$ such that $g = D^2 \Psi$ in $\nabla^+$-flat coordinates.

\item A \emph{Legendre bundle} $(H, \langle\!\langle \cdot, \cdot \rangle\!\rangle, \nabla^+, \nabla^-, \Psi)$ over $B$ in the sense of Definition~\ref{D:Legendre_Bundle}.
\end{enumerate}
\end{prop}

The correspondence is given explicitly as follows:

\begin{itemize}
\item Given a dually flat manifold with potential $\Psi$, define
\[
H = TB \oplus T^*B
\]
with its natural pairing. Let $\nabla^+$ be the given flat connection on $TB$, $\nabla^-$ the dual flat connection on $T^*B$, and retain $\Psi$ as the potential. This defines a Legendre bundle.

\item Conversely, given a Legendre bundle, define a Riemannian metric on $B$ by
\[
g(X,Y) = \langle\!\langle X, d\Psi(Y) \rangle\!\rangle.
\]
Using the identifications $H^+ \cong TB$ and $H^- \cong T^*B$, the connections $\nabla^+$ and $\nabla^-$ induce affine connections on $TB$ and $T^*B$. This equips $B$ with the structure of a dually flat manifold with potential $\Psi$.
\end{itemize}

These constructions are mutually inverse and preserve all the structure.

\begin{proof}
1. Assume $(B, g, \nabla^+, \nabla^-, \Psi)$ is a dually flat manifold with strictly convex potential $\Psi$. Set $H = TB \oplus T^*B$
and define a symmetric bilinear form by
\[
\langle\!\langle X \oplus \alpha,\; Y \oplus \beta \rangle\!\rangle
= \alpha(Y) + \beta(X).
\]
Then
\[
\langle\!\langle H^+, H^+ \rangle\!\rangle = 0,
\qquad
\langle\!\langle H^-, H^- \rangle\!\rangle = 0,
\]
and for $X \in H^+$, $\alpha \in H^-$,
\[
\langle\!\langle X, \alpha \rangle\!\rangle = \alpha(X).
\]

Take $\nabla^+$ to be the given flat connection on $TB \cong H^+$ and $\nabla^-$ the dual flat connection on $T^*B \cong H^-$. Define
\[
\mathscr{L}_{\Psi} : H^+ \to H^-,
\qquad
\mathscr{L}_{\Psi}(Y) = g(Y,\cdot).
\]
In $\nabla^+$-flat coordinates $(\theta^i)$, this gives
\[
\mathscr{L}_{\Psi}(\partial_j) = (\partial_i \partial_j \Psi)\, d\theta^i.
\]
All axioms of Definition~\ref{D:Legendre_Bundle} are satisfied, hence $(H,\langle\!\langle \cdot,\cdot \rangle\!\rangle,\nabla^+,\nabla^-,\Psi)$ is a Legendre bundle.

\medskip

2. Given a Legendre bundle $(H,\langle\!\langle \cdot,\cdot \rangle\!\rangle,\nabla^+,\nabla^-,\Psi)$, use the isomorphism $H^+ \cong TB$ to transfer $\nabla^+$ to a flat connection on $TB$. Define a bilinear form on $TB$ by
\[
g(X,Y) = \langle\!\langle X, \mathscr{L}_{\Psi}(Y) \rangle\!\rangle.
\]

In $\nabla^+$-flat coordinates $(\theta^i)$, $H^+$ is spanned by $\{\partial_i\}$ and $H^-$ by $\{d\theta^i\}$, with
\[
\langle\!\langle \partial_i, d\theta^j \rangle\!\rangle = \delta_i^j.
\]
By definition of $\mathscr{L}_{\Psi}$,
\[
\mathscr{L}_{\Psi}(\partial_j) = (\partial_i \partial_j \Psi)\, d\theta^i.
\]
Hence
\[
g(\partial_i, \partial_j)
= \langle\!\langle \partial_i, \mathscr{L}_{\Psi}(\partial_j) \rangle\!\rangle
= \langle\!\langle \partial_i, (\partial_k \partial_j \Psi)\, d\theta^k \rangle\!\rangle
= (\partial_k \partial_j \Psi)\, \delta_i^k
= \partial_i \partial_j \Psi.
\]

Thus $g = D^2 \Psi$ is a Hessian metric. Since $\nabla^+$ is flat and $\nabla^-$ satisfies the duality condition with respect to $\langle\!\langle \cdot,\cdot \rangle\!\rangle$, it follows that $\nabla^-$ is the connection dual to $\nabla^+$ with respect to $g$. Therefore $(B,g,\nabla^+,\nabla^-,\Psi)$ is a dually flat manifold.

\end{proof}

\subsection{Examples of Classical Legendre bundles}

\subsubsection{Exponential Families} Let $(\Omega, \Sigma)$ be a measurable space equipped with a family of probability measures $\mathfrak{P} = \{P_\theta\}_{\theta \in \Theta}$
dominated by a $\sigma$-finite (positive) measure $\mu$. Hence each $P_\theta$ is absolutely continuous with respect to $\mu$ and admits a (positive) density
$\rho_\theta = \frac{dP_\theta}{d\mu}\geq 0,$ such that $\int_{\Omega} \rho_\theta(\omega)\, \mu(d\omega)=1.$
The set $\{\rho_{\theta} \in L^1_+(\mu): \, \theta\in \Theta\}$ forms a subset of the positive cone of $L^1(\mu)$.

Recall that a statistic is a measurable map $T:({\Omega},\Sigma)\to ({\Upsilon},\Sigma')$, where $({\Upsilon},\Sigma')$ is a measurable space. 
A sub-$\sigma$-algebra $\mathscr{G} \subset \Sigma$ is said to be \emph{sufficient} for the family $\mathfrak{P}$ if for every set $A \in \Sigma$ there exists a $\mathscr{G}$-measurable function $f_A : {\Omega} \to [0,1]$ such that for every $\theta \in \Theta$ and every $G \in \mathscr{G}$,
\[
\int_G f_A \, dP_\theta = P_\theta(A \cap G).
\]
In other words, $f_A$ is a version of the conditional probability $P_\theta(A \mid \mathscr{G})$ that does not depend on $\theta$.

A statistic $T : ({\Omega}, \Sigma) \to ({\Upsilon}, \Sigma')$ is sufficient for the family $\mathfrak{P}$ if the $\sigma$-algebra $\sigma(T) = \{T^{-1}(B) : B \in \Sigma'\}$ is sufficient for $\mathfrak{P}$.

Equivalently, for every $A \in \Sigma$ there exists a $\sigma(T)$-measurable function $\varphi_A : \Omega \to [0,1]$ such that for every $\theta \in \Theta$ and every $B \in \Sigma'$,
\[
\int_{T^{-1}(B)} \varphi_A \, dP_\theta 
= P_\theta\bigl(A \cap T^{-1}(B)\bigr).
\]
The function $\varphi_A$ is then called a version of the conditional probability of $A$ given $T$ that is independent of $\theta$.

Let $T$ be the sufficient statistic, taking values in $\mathbb{R}^n$ i.e. $T\colon {\Omega}\to \mathbb{R}^n$.
We assume that we have an exponential family with respect to a measure $\mu$ and no carrier term, which implies that the probability density $\rho_{\theta}(\omega)$ can be expressed in a canonical form as follows: 
\[
\rho_{\theta}(\omega) = \exp\bigl(\langle \theta, T(\omega) \rangle - \Psi(\theta)\bigr), 
\qquad \theta \in \Theta,
\]
where the parameter space
\[
\Theta := \left\{ \theta \in \mathbb{R}^n \;\middle|\; 
\Psi(\theta) := \log \int_{\Omega} \exp\bigl(\langle \theta, T(\omega) \rangle\bigr)\, \mu(d\omega) < \infty \right\}
\]
is an open convex subset of $E = \mathbb{R}^n$, and $\Psi$ is the log-partition function. 
We now exhibit the Legendre duality structure (and thus the structure of a Legendre bundle). The function $\Psi$ is smooth and strictly convex on $\Theta$, hence $(\Theta, \Psi) \in \mathbf{CPS}$. The Hessian metric
\[
g_{ij}(\theta) = \partial_i \partial_j \Psi(\theta)
\]
is the Fisher--Rao metric. The dual coordinates are defined by
\[
\eta_i := \partial_i \Psi(\theta).
\]
Provided differentiation under the integral sign is justified, these admit the representation
\[
\eta_i = \mathbb{E}_\theta[T_i],
\qquad
\mathbb{E}_\theta[T_i]
= \int_{\Omega} T_i(\omega)\, \rho_{\theta}(\omega)\, \mu(d\omega),
\]
where $\mathbb{E}_\theta$ denotes expectation with respect to $P_\theta$.

The Legendre conjugate $\Psi^*(\eta)=\sup\limits_\theta\{\, \langle \theta,\eta\rangle -\Psi(\theta)\, \}$ corresponds (up to sign conventions) to the entropy, and the Fenchel--Young identity reads
\[
\Psi(\theta) + \Psi^*(\eta) = \langle \theta, \eta \rangle.
\]

\begin{prop}
Every exponential family carries a natural Legendre bundle structure, with
\[
H^+ = T\Theta, \qquad H^- = T^*\Theta,
\]
flat connections $\nabla^+$ and $\nabla^-$ induced by the affine structures on $\Theta$ and $\Theta^*$ respectively, and potential $\Psi$ given by the log-partition function.
\end{prop}

\begin{proof}
For an exponential family, the log‑partition function $\Psi$ is strictly convex on the open convex set $\Theta\subset \mathbb{R}^n$. The Hessian metric $g = D^2 \Psi$ defines a Riemannian metric and the standard affine coordinates $\theta^i$ give a flat connection $\nabla^+$ on $T\Theta$.
The dual coordinates $\eta_i=\partial\Psi$ define a flat connection $\nabla^-$ on $T^*\Theta$. This dual connection is uniquely determined by the duality condition $
(Z \cdot g)(X,Y) = g(\nabla^+_Z X, Y) + g(X, \nabla^-_Z Y)$ with respect to the Hessian metric $g$.

 Now construct the Legendre bundle by taking $H = T\Theta \oplus T^*\Theta$
equipped with the canonical pairing
\[
\langle\!\langle X \oplus \alpha, Y \oplus \beta \rangle\!\rangle = \alpha(Y) + \beta(X).
\]

By taking $H^+ = T\Theta, \quad H^- = T^*\Theta,$ it follows that the connections $\nabla^+$ and $\nabla^-$ are exactly the flat connections described above, and the potential is $\Psi$. The Legendre morphism is the bundle map
\[
\mathscr{L}_\Psi : H^+ \to H^-, \qquad \mathscr{L}_\Psi(\partial_i) = (\partial_i \partial_j \Psi)\, d\theta^j
\]
in $\theta$-coordinates, extended linearly.  

By Proposition~\ref{P:2.4}, these data satisfy all the axioms of a Legendre bundle.\end{proof}

\subsection{Para-K\"ahler structures on Legendre bundles}

\begin{dfn}
    Let $B$ be a smooth manifold. A \emph{para-Kähler vector bundle} over $B$ is a quadruple $(H, J, \boldsymbol{\omega}^S, \nabla)$ where:

\begin{itemize}
    \item $H$ is a real vector bundle over $B$.
    \item $J$ is a vector bundle endomorphism of $H$ such that
    \[
    J^2 = \mathrm{id}_H,
    \]
    and the subbundles
    \[
    H^+ = \ker(J - \mathrm{id}_H), \qquad H^- = \ker(J + \mathrm{id}_H)
    \]
    have equal rank.
    \item $\boldsymbol{\omega}^S$ is a non-degenerate skew-symmetric bilinear form on the fibres of $H$ (i.e., a symplectic form on each fibre) satisfying
    \[
    \boldsymbol{\omega}^S(JX, JY) = -\,\boldsymbol{\omega}^S(X, Y) \quad (\forall X, Y \in H).
    \]
    \item $\nabla$ is a flat connection on $H$ (i.e., its curvature vanishes) such that
    \[
    \nabla J = 0 \quad \text{and} \quad \nabla \boldsymbol{\omega}^S = 0.
    \]
\end{itemize}

The bundle is said to be para-Kähler because the pair $(J, \boldsymbol{\omega}^S)$ defines a para-complex structure on each fibre that is compatible with the symplectic form, and the connection is flat and preserves both structures.
\end{dfn}

The classical Legendre bundle carries a natural paracomplex structure: it carries a bundle endomorphism $J : H \to H$
such that $J^2 = \mathrm{id}$. It is defined by \begin{equation}\label{E:para} J|_{H^+} = +\mathrm{id}, \qquad J|_{H^-} = -\mathrm{id},\end{equation}
so that the decomposition $H = H^+ \oplus H^-$ is exactly the eigenspace decomposition of $J$.  

This endows $H$ with the structure of a \emph{para-Kähler vector bundle}~\cite{CM} when combined with the canonical pairing $\langle\!\langle \cdot, \cdot \rangle\!\rangle$: the bilinear form
\begin{equation}\label{E:symplectic}
    \boldsymbol{\omega}^S(X,Y) = \langle\!\langle JX, Y \rangle\!\rangle
\end{equation}
is a symplectic form on each fibre, and $J$ is parallel with respect to the flat connections $\nabla^+$ and $\nabla^-$, in the sense that the induced connection on $H$ preserves the splitting $H = H^+ \oplus H^-$.

\begin{prop}\label{P:ParaKahler}
The (Legendre) bundle $(H,J,\boldsymbol{\omega}^S)$ is a para‑Kähler vector bundle over $B$. 
\end{prop}
    \begin{proof}
Given the decomposition $H = H^+ \oplus H^-$, define $X = X^+ \oplus \alpha$ and $Y = Y^+ \oplus \beta,$ where $X^+, Y^+ \in \Gamma(H^+)$ and $\alpha, \beta \in \Gamma(H^-)$. Then,
by definition of $J$ we obtain \[
JX = X^+ \oplus (-\alpha).
\]
Using the pairing
\[
\langle\!\langle X^+, \beta \rangle\!\rangle = \beta(X^+)
\quad \text{and} \quad
\langle\!\langle \alpha, Y^+ \rangle\!\rangle = \alpha(Y^+),
\]
we define  $$\boldsymbol{\omega}^S(X,Y)=\langle\!\langle 
 J(X \oplus \alpha),\, Y \oplus \beta \rangle\!\rangle
= \langle\!\langle X^+ \oplus (-\alpha),\, Y^+ \oplus \beta \rangle\!\rangle$$ \[
= \langle\!\langle X^+, Y^+ \rangle\!\rangle
+ \langle\!\langle X^+, \beta \rangle\!\rangle
+ \langle\!\langle -\alpha, Y^+ \rangle\!\rangle
+ \langle\!\langle -\alpha, \beta \rangle\!\rangle
\]
\[
= 0 + \beta(X^+) + \bigl(-\alpha(Y^+)\bigr) + 0
= \beta(X^+) - \alpha(Y^+).
\] which is the canonical symplectic form on $T^*B$ (when identifying $H^+ \cong TB$, $H^- \cong T^*B$).

The natural connection on $H$ is the direct sum $\nabla = \nabla^+ \oplus \nabla^-$, defined by
\[
\nabla(X^+ \oplus \alpha) = (\nabla^+ X^+) \oplus (\nabla^- \alpha).
\]
This connection preserves the splitting, hence it commutes with $J$ (since $J$ is constant on each factor). Moreover, because the pairing $\langle\!\langle \cdot, \cdot \rangle\!\rangle$ is flat with respect to $\nabla^+ \oplus \nabla^-$ (by the duality condition), the symplectic form $\boldsymbol{\omega}^S$ is also flat. Thus $(H, J, \boldsymbol{\omega}^S, \nabla)$ is a flat para-Kähler vector bundle.

    \end{proof}

\section{Hessian Quantum Field Theory}
In what follows we introduce a definition for specific classes of QFTs. We name them Hessian QFTs, after their geometric properties. 

Let $k$ be a field of characteristic zero ($\mathbb{R}$ or $\mathbb{C}$) and set $D = \operatorname{Spec} k[[u]]$, the formal disk with coordinate $u$. 

\smallskip 

All geometric objects considered below are defined over the formal power series ring $k[[u]]$; this means that their local descriptions involve formal power series in $u$ with coefficients in smooth functions on the base manifold. This is precisely the unifying bridge: the same geometric object (a family of Legendre bundles) describes both statistical models (at $u=0$) and QFTs (with $u$ as a formal parameter). The formal approach avoids convergence issues.

\smallskip 

This  draws inspiration from the framework of $F$-bundles defined in \cite{K,KKP}.

\begin{dfn}[Hessian QFT]\label{D:gQFT}
A Hessian quantum field theory is a quantum field theory satisfying the following conditions:
\begin{enumerate}
    \item Its coupling space $\cT$ is an open convex subset of $\mathbb{R}^n$ (or more generally a smooth manifold with global affine coordinates).
    \item The free energy (log-partition function) $F(t,u)$, is a formal power series in a deformation parameter $u$ (often identified with $\hbar$ or the genus expansion) with coefficients in $C^\infty(\cT)$:
    \[
        F(t,u) = F_0(t) + u F_1(t) + u^2 F_2(t) + \cdots
    \]
     and is strictly convex for each fixed $u$ in the formal sense.

    \item The Hessian
\[
    g_{ij}(t,u) = \partial_i \partial_j F(t,u)
\]
defines a positive-definite metric (the \emph{Zamolodchikov metric} in the conformal case) which, together with the affine structure on $\cT$ and its dual, endows $\cT$ with the structure of a dually flat (Hessian) manifold for each fixed $u$.
\end{enumerate}
\end{dfn}

In the case $u=0$, the definition recovers the classical dually flat structure of exponential families in information geometry, with $\cT$ as the parameter space.

\subsection{Restricted Class and Formal Setup}\label{sec:restricted_class}
The coupling space $\mathcal{T}$ carries a natural notion of distinguishability between states. Depending on the context—probability distributions, density matrices, or quantum fields—this is measured respectively by the Fisher, Bures, or Zamolodchikov metric. These metrics arise from a single geometric object, the Provost--Vall\'ee \emph{quantum geometric tensor} (QGT). For a smooth family of normalized quantum states $|\psi(\theta)\rangle$, the QGT is
\[
\mathscr{Q}_{ij}(\theta) 
= \langle \partial_i \psi \mid \partial_j \psi \rangle 
- \langle \partial_i \psi \mid \psi \rangle \, \langle \psi \mid \partial_j \psi \rangle.
\]
This reduces to the Fisher metric for classical probability distributions and to the Zamolodchikov metric in conformal field theory, while its imaginary part gives the Berry curvature. In the setting of a Hessian QFT, the metric $g_{ij}=\partial_i\partial_jF(t,u)$ coincides with the real part of the QGT, thereby providing a unifying geometric interpretation.

Instances of theories that can be realised as Hessian QFTs include the following:
\begin{itemize}
    \item \textbf{Topological field theories (TFTs)} whose free energy $F(t,u)$ is a formal power series with a positive‑definite Hessian at the point of expansion (e.g., free TFTs or Gaussian theories). In many TFTs the free energy is polynomial; strict convexity is not automatic, but the definition only requires convexity in the formal sense, which can be arranged by considering a sufficiently small neighbourhood of a point where the Hessian is positive definite.
    \item \textbf{Conformal field theories (CFTs) with exactly marginal deformations}. For such theories the space of couplings is a Frobenius manifold (see \cite{D} 1996; \cite{CV} 1993), which is a dually flat (Hessian) manifold. In particular, the free energy $F(t,u)$ is a formal power series that is strictly convex in the formal sense, and the metric coincides with the Zamolodchikov metric. Hence these CFTs are Hessian QFTs.
    \item \textbf{Zero‑dimensional QFT} ($M = \{\mathrm{pt}\}$), which recovers the exponential families of information geometry. Here $u$ is absent (or $u=0$) and the free energy is the cumulant generating function, which is strictly convex by definition.
\end{itemize}

\subsection{Families of Legendre Bundles}\label{sec:families}

We introduce now the family of Legendre bundles over $B\times D$ where $D = \operatorname{Spec} k[[u]]$ is the formal disk with coordinate $u$, where $k$ is either \(\mathbb{R}\) or \(\mathbb{C}\). Notice that all objects are smooth in the base directions and formal in \(u\). This construction introduces a deformation parameter $u$. At $u=0$ we recover the classical structure (the classical Legendre bundle from definition~\ref{D:Legendre_Bundle}). Higher‑order terms in $u$ encode quantum corrections.

This formalism allows us to treat certain QFTs (those whose coupling space admits a formal convex potential) as {\it quantum deformations} of a classical dually flat manifold, i.e., as a {\it family of Legendre bundles}.

\subsection{Family of Legendre Bundles}\label{S:F-Legendrebundle}
\begin{dfn}[Family of Legendre bundles]\label{D:FLB}
A \emph{family of Legendre bundles} over $B\times D$ is a quintuple
\[
\bigl( H,\; \langle\!\langle\cdot,\cdot\rangle\!\rangle,\; \nabla^+,\; \nabla^-,\; \Psi \bigr)
\]
where
\begin{itemize}
\item $H = H^+ \oplus H^-$ is a vector bundle over $B\times D$ with isomorphisms 
  $H^+ \cong \pi_B^*TB$ and $H^- \cong \pi_B^*T^*B$ that are independent of $u$;
\item $\langle\!\langle\cdot,\cdot\rangle\!\rangle$ is a non‑degenerate symmetric bilinear form on $H$ such that 
  $\langle\!\langle H^+,H^+\rangle\!\rangle = 0$, $\langle\!\langle H^-,H^-\rangle\!\rangle = 0$,
  and the induced pairing $H^+\otimes H^-\to \mathcal{O}_{B\times D}$ is the canonical evaluation;
\item $\nabla^+$ and $\nabla^-$ are flat connections on $H^+$ and $H^-$ respectively, 
  satisfying the duality condition
  \[
  Z\langle\!\langle X,\alpha\rangle\!\rangle = \langle\!\langle \nabla^+_Z X,\alpha\rangle\!\rangle + \langle\!\langle X,\nabla^-_Z\alpha\rangle\!\rangle
  \]
  for all sections $X$ of $H^+$, $\alpha$ of $H^-$ and all vector fields $Z$ on $B\times D$ (in particular for $\partial_u$);
\item $\Psi$ is a formal power series in $u$ with coefficients in $C^\infty(B)$,
  \[
  \Psi(\cdot,u) = \psi(\cdot) + u\psi_1(\cdot) + u^2\psi_2(\cdot) + \cdots,
  \]
  such that for each fixed $u$ (in the formal sense) the Hessian $\partial_i\partial_j\Psi$ defines a positive‑definite metric on $B$ (such that the Hessian of the constant term $\psi$ is positive‑definite; the higher‑order terms $\psi_1,\psi_2,\dots$ are arbitrary formal power series).
\end{itemize}
The data are required to satisfy the conditions of Definition~\ref{D:Legendre_Bundle} (the classical Legendre bundle) for each fixed $u$.
\end{dfn}

\begin{remark}

At $u=0$ the family reduces to the classical Legendre bundle $\bigl( H_0,\langle\!\langle\cdot,\cdot\rangle\!\rangle_0,\nabla^+_0,\nabla^-_0,\psi \bigr)$ over $B$, where $H_0 = H|_{u=0}$ and $\psi$ is the original convex potential. The higher‑order terms $\psi_1,\psi_2,\dots$ encode deformations of the classical dually flat structure.
\end{remark}

\smallskip

We proceed to considerations linking the para-Kähler structure and the family of Legendre bundles. 

\smallskip 

\begin{prop}\label{P:FamilyParaKahler}[Deformation of the para-Kähler structure]~

    Let $\bigl( H,\; \langle\!\langle\cdot,\cdot\rangle\!\rangle,\; \nabla^+,\; \nabla^-,\; \Psi \bigr)$ be a family of Legendre bundles over $B\times D$ as in Definition~\ref{D:FLB}. Then for each fixed $u$ (in the formal sense), the restriction $H_u = H|_{B\times\{u\}}$ together with the induced data $(J,\boldsymbol{\omega}^S_u,\nabla_u)$ defined by  

\begin{itemize}
    \item $J = +\mathrm{id}$ on $H^+_u$ and $J = -\mathrm{id}$ on $H^-_u$;  
    \item $\boldsymbol{\omega}^S_u(X,Y) = \langle\!\langle JX, Y\rangle\!\rangle$;  
    \item $\nabla_u = \nabla^+_u \oplus \nabla^-_u$
\end{itemize}  

is a \emph{para-Kähler vector bundle} over $B$.  

In particular, the family $(H,\langle\!\langle\cdot,\cdot\rangle\!\rangle,\nabla^+,\nabla^-,\Psi)$ is a \emph{formal deformation} of the para-Kähler bundle at $u=0$ (the classical Legendre bundle) with deformation parameter $u$.  
\end{prop}
   
\begin{proof}
    
For each fixed $u$, the data satisfy the classical Legendre bundle conditions (by definition of a family). As shown in Proposition~\ref{P:ParaKahler} (the para-Kähler structure of the classical Legendre bundle), $(H_u,J,\boldsymbol{\omega}^S_u,\nabla_u)$ fulfills all axioms of a para-Kähler vector bundle:  
\begin{itemize}
    \item $J^2 = \mathrm{id}$ and the eigenbundles $H^\pm_u$ have equal rank;  
    \item $\boldsymbol{\omega}^S_u$ is non-degenerate, skew-symmetric, and satisfies $\boldsymbol{\omega}^S_u(JX,JY)=-\boldsymbol{\omega}^S_u(X,Y)$;  
    \item $\nabla_u$ is flat (since $\nabla^+_u$ and $\nabla^-_u$ are flat) and preserves $J$ (because it preserves the splitting) and $\boldsymbol{\omega}^S_u$ (because $\nabla_u$ preserves $\langle\!\langle\cdot,\cdot\rangle\!\rangle$ by the duality condition).
\end{itemize}  

Hence $(H_u,J,\boldsymbol{\omega}^S_u,\nabla_u)$ is para-Kähler. The dependence on $u$ is formal, so the family is a formal deformation of the $u=0$ fibre.
\end{proof}

\subsection{The Extended Bundle}
To incorporate the deformation parameter into the connection, we extend $H$ by the line bundle $L$ spanned by the vector field $\partial_u$ (tangent to $D$). 
Define
\[
\widetilde{H} = H \oplus L,
\]
which is a vector bundle over $B\times D$ of rank $2\dim B+1$. 

\begin{remark}
  $\widetilde{H}$ is a convenient geometric setting to incorporate the deformation parameter. However, it does not itself carry a para‑Kähler structure.
\end{remark}

\section{Family of Legendre Bundles $\&$ Hessian QFTs}\label{S:Main}
In what follows we work in the category of formal manifolds over the formal disk $D = \operatorname{Spec} k[[u]]$, where $u$ is a deformation parameter. In this framework, the metric is exactly the Zamolodchikov metric (when the theory is conformal), and the flat connections are those induced by the flat coordinates. The deformation parameter $u$ then controls the quantum deformation of the classical dually flat structure.

\subsection{Statement}
With these assumptions, given the coupling space $\mathcal{T}$ of a Hessian QFT (with the formal parameter $u$) carries a natural \emph{family of Legendre bundles} over $\mathcal{T} \times D$, as defined in Definition~\ref{D:FLB}. More concretely we prove the following: 

 \begin{thm}[Family of Legendre bundles and QFTs]
\label{T:1}
Let $\cT$ be the coupling space of a Hessian QFT (defined as in Definition~\ref{D:gQFT}) and let $F(t,u)$ be its free energy (a formal power series in $u$ with coefficients in $C^\infty(\cT)$). Assume that for each fixed $u$ (in the formal sense), $F(\cdot,u)$ is strictly convex and the Hessian metric
$g_{ij}(t,u) = \partial_i \partial_j F(t,u)$ is positive definite.

Then the coupling space $\cT$ carries a natural family of Legendre bundles over $\cT \times D$ with $D = \mathrm{Spec}\, k[[u]],$ given by:
\[
\bigl( H,\; \langle\!\langle\cdot,\cdot\rangle\!\rangle,\; \nabla^+,\; \nabla^-,\; \Psi \bigr),
\]
where:
\begin{itemize}
    \item $H = H^+ \oplus H^-$ with $H^+ = \pi_{\mathcal{T}}^*T\mathcal{T}$, $H^- = \pi_{\mathcal{T}}^*T^*\mathcal{T}$;
    \item $\langle\!\langle\cdot,\cdot\rangle\!\rangle$ is the canonical evaluation between tangent and cotangent vectors, extended to vanish on $H^+$ and $H^-$ individually;
    \item $\nabla^+$ and $\nabla^-$ are the flat connections induced by the affine coordinates on $\mathcal{T}$ and on the dual space $\mathcal{T}^*$, respectively;
    \item $\Psi(t,u) = F(t,u)$ is the free energy.
\end{itemize}
\end{thm}
The data satisfy all conditions of Definition~\ref{D:FLB} for each fixed $u$, and at $u=0$ the family reduces to the classical Legendre bundle of the tree‑level theory. In the zero‑dimensional case $M=\{\mathrm{pt}\}$, this recovers precisely the Legendre bundle of an exponential family.
\begin{proof}
By definition of a Hessian QFT, for each $u$ the pair $(\cT, F(\cdot,u))$ is a dually flat manifold.  
By Proposition~\ref{P:2.4}, each such dually flat manifold yields a Legendre bundle over $\cT$ with 
\[
    H = T\cT \oplus T^*\cT,
\] 
the canonical pairing, flat connections $\nabla^+$ and $\nabla^-$, and potential $F(\cdot,u)$.  

Because the constructions are independent of $u$ (the bundle and connections do not involve $u$-derivatives), these data assemble into a family over 
\[
    \cT \times D, \qquad D = \mathrm{Spec}\, k[[u]],
\] 
with potential 
\[
    \Psi(t,u) = F(t,u).
\] 
This family satisfies therefore Definition~\ref{D:FLB} by construction.

 \end{proof}

We provide the following corollary, based on the previous results.  

\begin{cor}\label{C:QFTParaKahler}
    Every Hessian QFT provides a family of para-Kähler vector bundles parameterized by the formal deformation parameter $u$, with the classical case $u=0$ corresponding to the tree-level limit.
\end{cor}

The tree-level limit here refers to setting the formal deformation parameter $u = 0$. In quantum field theory, $u$ is often identified with $\hbar$ (or the genus expansion parameter). The tree-level approximation is the classical limit where quantum loops are neglected; it corresponds to the leading term $F_0(t)$ in the expansion
\[
F(t,u) = F_0(t) + u F_1(t) + u^2 F_2(t) + \cdots.
\]
At $u = 0$, the free energy reduces to $F_0(t)$, which is exactly the strictly convex potential of a dually flat (Hessian) manifold — the structure underlying classical exponential families in information geometry. Hence, the tree-level limit recovers the statistical (or classical) case.

\section*{Conclusion}

We have introduced the notion of a \emph{Legendre bundle}---a geometric structure that encapsulates the essential duality of dually flat (Hessian) manifolds. By demonstrating that both exponential families in information geometry and a large class of quantum field theories (which we named Hessian QFTs) naturally realise this structure, we have built a rigorous bridge between two previously separate domains. The Legendre bundle provides a common language: at $u=0$ it recovers the classical geometry of statistical inference (Fisher metric, exponential families), while the formal parameter $u$ encodes quantum corrections, connecting with the free energy of topological and conformal field theories.

The formal power series $\Psi(t,u)=F(t,u)$ is strongly reminiscent of generating functions that appear in the theory of D‑modules and deformation quantization, where a formal parameter deforms a classical commutative structure into a non‑commutative or flat connection‑based one. Notably, the appearance of $\Psi(t,u)$ as a generating function for a family of Legendre bundles also echoes the role of the genus expansion in Gromov–Witten theory, where the free energy encodes invariants such as gravitational descendants. In this context, the Legendre bundle framework may provide a new geometric perspective on the formal deformation structures underlying the constructions in~\cite{K}, as it draws inspiration from the $F$-bundles defined therein.

It is natural to ask whether, under additional strong assumptions (e.g., flatness of the Hessian metric), the family of Legendre bundles could give rise to an $F$-bundle structure (see \cite{K,KKP}). However, this is not the case in general and is left for future work. Nevertheless, the family of Legendre bundles over the formal disk $D = \operatorname{Spec} k[[u]]$ can be viewed as a geometric realisation of a formal deformation of a dually flat manifold. The flat connections $\nabla^+_u$ and $\nabla^-_u$ on $H^+$ and $H^-$ are naturally adapted to the Legendre duality and may serve as a starting point for a quantisation of Hessian manifolds via deformation quantization.

In summary, the Legendre bundle framework reveals that Legendre duality is not just a common computational tool but a unifying geometric principle  . By making this principle explicit and rigorous, we provide a new perspective that we hope will inspire further cross‑disciplinary research.

\Addresses


\begin{thebibliography}{99}
\bibitem{CV} S. Cecotti and C. Vafa, {\sl Topological–anti-topological fusion}, Nucl. Phys. B {\bf 367} (1993), pp. 359–461.
\bibitem{Ch} N. N. Chentsov {\sl Statistical Decision Rules and Optimal Inference} Translations of Mathematical Monographs, Vol. {\bf 53}
American Mathematical Society, (1982). 
\bibitem{CCN} N. Combe, P. Combe and H. Nencka {\sl Exploring Information Geometry: Recent Advances and Connections to Topological Field Theory}  Frontiers in Mathematics, Birkhäuser Cham 228, (to appear) June 2026.
\bibitem{CM} N. Combe and Y. I. Manin, {\sl F-manifolds and geometry of information,}
Bull. Lond. Math. Soc. {\bf 52} (2020), no. 5, pp. 777–792.
\bibitem{D} B. Dubrovin, {\sl Geometry of 2D topological field theories}, in Integrable Systems and Quantum Groups, Lecture Notes in Math. {\bf 1620}, Springer, (1996), pp. 120–348.
\bibitem{K} M. Kontsevich, {\sl  Birational Invariants from Gromov-Witten Theory}, Lectures I-IV at IHES, (Nov.- Dec. 2023).

\bibitem{KKP}  L. Katzarkov, M. Kontsevich, T. Pantev,
{\sl Hodge theoretic aspects of mirror symmetry},
in From Hodge Theory to Integrability and TQFT,
Proceedings of Symposia in Pure Mathematics, Vol.{\bf  78,}
American Mathematical Society, 2008, pp. 87–174.
\end{thebibliography}
\end{document}